\documentclass{amsart}
\usepackage[utf8]{inputenc}
\usepackage{amsmath, amssymb, bbm}
\usepackage{amsthm}
\usepackage{mathtools}
\usepackage{bbm}
\usepackage{blkarray}
\usepackage[matrix,arrow,curve]{xy}
\usepackage{color}
\usepackage{xcolor}
\usepackage{graphicx}
\definecolor{darkblue}{rgb}{0,0,0.6}
\usepackage[ocgcolorlinks,colorlinks=true, citecolor=darkblue, filecolor=darkblue, linkcolor=darkblue, urlcolor=darkblue]{hyperref}
\usepackage[capitalize,noabbrev]{cleveref}
\usepackage{comment}
\usepackage{adjustbox}
\usepackage{enumerate}

\usepackage{tikz}
\usetikzlibrary{
  cd,
  calc,
  positioning,
  arrows,
  decorations.pathreplacing,
  decorations.markings,
}

\usepackage[dvips,%
    includehead,%
    includefoot,%
    nomarginpar,%
    lmargin=1.2in,%
    rmargin=1.2in,%
    tmargin=1.4in,%
    bmargin=1.4in,%
  ]{geometry}

\newcounter{commentcounter}

\newcommand{\ignore}[1]{}
\renewcommand{\epsilon}{\varepsilon}
\renewcommand{\phi}{\varphi}
\newcommand{\bbZ}{\mathbbm{Z}}
\newcommand{\Z}{\mathbbm{Z}}

\newcommand{\R}{\mathbbm{R}}
\newcommand{\CP}{\mathbb{CP}}
\newcommand{\RP}{\mathbb{RP}}

\newcommand{\wt}{\widetilde}
\newcommand{\ol}{\overline}
\newcommand{\Sq}{\mathrm{Sq}}
\newcommand{\red}{\mathrm{red}}

\newcommand{\sm}{\setminus}
\newcommand{\pt}{\mathrm{pt}}
\newcommand{\ter}{\mathfrak{ter}}
\newcommand{\msec}{\mathfrak{sec}}
\newcommand{\pri}{\mathfrak{pri}}
\newcommand{\SC}{\mathcal{SC}}
\newcommand{\RC}{\mathcal{RC}}

\newcommand{\wh}{\widehat}
\newcommand{\ba}{\begin{array}}
\newcommand{\ea}{\end{array}}

\newcommand{\imra}{\looparrowright}
\newcommand{\RPT}{\mathbb{RP}^2}

\newcommand{\spin}{\mathrm{Spin}}

\DeclareMathOperator{\Aut}{Aut}

\DeclareMathOperator{\Hom}{Hom}

\DeclareMathOperator{\im}{im}

\DeclareMathOperator{\Rad}{Rad}

\DeclareMathOperator*{\colim}{colim}
\DeclareMathOperator{\Arf}{Arf}

\DeclareMathOperator{\cs}{cs}

\numberwithin{equation}{section}
\newtheorem{thm}[equation]{Theorem}
\newtheorem{pre-thm}[equation]{Pre-Theorem}

\newtheorem{question}[equation]{Question}

\newtheorem{lemma}[equation]{Lemma}

\theoremstyle{definition}

\newtheorem{example}[equation]{Example}
\newtheorem{defi}[equation]{Definition}
\newtheorem{definition}[equation]{Definition}

\newtheorem{remark}[equation]{Remark}

\crefname{lemma}{Lemma}{Lemmas}
\crefname{section}{Section}{Sections}
\crefname{definition}{Definition}{Definitions}
\crefname{defi}{Definition}{Definitions}
\crefname{prop}{Proposition}{Propositions}
\crefname{theorem}{Theorem}{Theorems}
\crefname{thm}{Theorem}{Theorems}
\crefname{cor}{Corollary}{Corollaries}
\crefname{remark}{Remark}{Remarks}

\title[The Kervaire--Milnor invariant in the stable classification of spin $4$-manifolds]{The Kervaire--Milnor invariant in the stable classification\\ of spin $4$-manifolds}

\author{Daniel Kasprowski}
\address{School of Mathematical Sciences, University of Southampton, United Kingdom}
\email{d.kasprowski@soton.ac.uk}

\author{Mark Powell}
\address{School of  Mathematics and Statistics, University of Glasgow, United Kingdom}
\email{mark.powell@glasgow.ac.uk}

\author{Peter Teichner}
\address{Max Planck Institut f\"{u}r Mathematik, Vivatsgasse 7, 53111 Bonn, Germany}
\email{teichner@mac.com}

\def\subjclassname{\textup{2020} Mathematics Subject Classification}
\expandafter\let\csname subjclassname@1991\endcsname=\subjclassname

\begin{document}

\begin{abstract}
We consider the role of the Kervaire--Milnor invariant in the classification of closed, connected, spin $4$-manifolds, typically denoted by $M$, up to stabilisation by connected sums with copies of $S^2 \times S^2$.  This stable classification is detected by a spin bordism group over the classifying space $B\pi$ of the fundamental group. Part of the computation of this bordism group via an Atiyah--Hirzebruch spectral sequence is determined by a collection of codimension two Arf invariants. We show that these Arf invariants can be computed by the Kervaire--Milnor invariant evaluated on certain elements of $\pi_2(M)$.  In particular this yields a new stable classification of spin $4$-manifolds with 2-dimensional fundamental groups, namely those for which $B\pi$ admits a finite 2-dimensional CW-complex model.
\end{abstract}

\maketitle

\section{Introduction}
Two smooth, closed, connected, oriented 4-manifolds $M$ and $N$ are called \emph{stably diffeomorphic} if there exist integers $m,n \in \mathbb{N}_0$ such that
\[M \#^m (S^2 \times S^2) \cong N \#^n (S^2 \times S^2),
\]
where $\cong$ denotes diffeomorphism.
We require that the diffeomorphism respects orientations, and we will always assume without comment that manifolds are \emph{smooth and connected}.
Note that we allow $m \neq n$, but $n-m = (\chi(M) - \chi(N))/2$, so by only considering 4-manifolds with the same Euler characteristic one can enforce $m=n$.

The problem of giving algebraic invariants that determine whether two 4-manifolds are stably diffeomorphic is the \emph{stable classification problem}.  For example, the isometry class of the equivariant intersection form on the second homotopy group, up to stabilisation by hyperbolic forms, is such an invariant.
Two closed, simply-connected 4-manifolds are stably diffeomorphic if and only if their intersection forms have the same parity and equal signatures.

%The stable classification of simply-connected 4-manifolds is detected by the signature of this intersection form.

Given an immersed $2$-sphere $S$ in $M$ with vanishing self-intersection number, the \emph{Kervaire--Milnor invariant} arises as a secondary obstruction to homotoping~$S$ to an embedding. Our aim in this article is to explain its role in the stable classification problem for spin 4-manifolds.
There is an additional condition, namely that $S$ is \emph{$\RPT$-characteristic}, under which the Kervaire--Milnor invariant $\tau(S) \in \Z/2$ is well-defined, depending only on the homotopy class $[S] \in \pi_2(M)$. We defer the precise definitions to \cref{sub-sec:review-KM}.

For every closed, oriented 4-manifold $M$ with $\pi= \pi_1(M)$ and 2-connected map $f\colon M\to B\pi$, the radical $\Rad(\lambda_{M}) \subseteq \pi_2(M)$ of the intersection form $\lambda_M$, which by definition is the kernel of the adjoint $\lambda_M^{\text{ad}} \colon \pi_2(M) \to \pi_2(M)^*$, is isomorphic to $H^2(\pi;\Z\pi)$. In more detail, the map $PD \circ f^* \colon H^2(\pi;\Z\pi) \to H_2(M;\Z\pi) \xrightarrow{\cong} \pi_2(M)$ is injective and has image $\ker \lambda_M^{\text{ad}} \subseteq \pi_2(M)$, by~\cite[Corollary~3.2]{HKT}.
Also, write $\red_2 \colon H^2(\pi;\Z\pi) \to H^2(\pi;\Z/2)$ for the map induced by modulo two augmentation, which we also denote by  $\red_2 \colon \Z\pi \to \Z/2$.
%reduction modulo two on the coefficients.
Define
\[\Sq :=(\Sq^2\circ \red_2) \colon  H^2(\pi;\Z\pi)\to H^4(\pi;\Z/2).\]
Note that $\Sq^2 \colon H^2(\pi;\Z/2) \to H^4(\pi;\Z/2)$ is the map $x \mapsto x \cup x$.

By restricting to the radical, $[S] \in \Rad(\lambda_{M})$, we guarantee vanishing self-intersection number.
For spin~$M$  we will show that restricting further to $PD(f^*(\ker \Sq )) \subseteq \Rad(\lambda_{M})$ ensures that $S$ is $\RPT$-characteristic, and thus that $\tau$ gives rise to a stable diffeomorphism invariant of spin $4$-manifolds, as follows.
We will restrict to spin $M$ for the remainder of the article.

\begin{thm}\label{thm1}
Let $M$ be a closed, spin 4-manifold with a $2$-connected map $f \colon M \to B\pi$.
\begin{enumerate}[(i)]
\item\label{thm1:item1} For every element $x \in \ker \Sq \subseteq H^2(\pi;\Z\pi)$, its image $PD(f^*(x)) \in \pi_2(M)$  has trivial self-intersection number and is $\RPT$-characteristic, so the Kervaire--Milnor invariant \[\tau(PD(f^*(x)))  \in \Z/2\] is well-defined.
\item\label{thm1:item2}  The induced map $\tau_{M,f} \colon \ker \Sq \to \Z/2$ factors through $\Z/2\otimes_{\Z\pi} \ker \Sq $.
\item\label{thm1:item3}  If a closed, spin $4$-manifold $N$ is stably diffeomorphic to $M$, then there exists a $2$-connected map $g \colon N \to B\pi$, so in particular $\pi_1(N) \cong \pi$, such that $\tau_{M,f} = \tau_{N,g}$ as maps $\ker \Sq \to \Z/2$.
  \end{enumerate}
\end{thm}

In \cref{thm:main-technical-thm-intro} below, we explain how $\tau_{M,f}$ appears in the general stable classification programme, in the case that $f$ factors through a 2-dimensional complex.
\cref{thm:main-technical-thm-intro} requires some more background in order to state, so for now we present its application in the case of geometrically 2-dimensional fundamental groups. A group $\pi$ is (geometrically) $d$-dimensional if $d$ is the least integer for which the classifying space $B\pi$ admits a finite $d$-dimensional CW-complex model. If $d < \infty$ then each $d$-dimensional group is torsion-free. Note that for a $2$-dimensional group $H^4(\pi;\Z/2)=0$ so $\ker \Sq = H^2(\pi;\Z\pi)$.

\begin{thm}\label{thm:geom-2-dim-gps}
  Let $\pi$ be a 2-dimensional group, let $M$ and $N$ be closed, spin 4-manifolds, and let $f\colon M\to B\pi$ be a $2$-connected map.
  \begin{enumerate}[(i)]
  \item\label{it:thm1.2.1} The map $\tau_{M,f}$ is a homomorphism, i.e.\ an element of $\Hom_{\Z\pi}(H^2(\pi;\Z\pi),\Z/2)$.
  \item The $4$-manifolds $M$ and $N$ are stably diffeomorphic if and only if
  \begin{enumerate}[(a)]
    \item\label{item:geom-2-dim-0} the signatures of $M$ and $N$ are equal, and
    \item\label{item:geom-2-dim-3}
   There exists a $2$-connected map $g\colon N\to B\pi$, so in particular $\pi_1(N) \cong \pi$, such that
    the Kervaire--Milnor invariants $\tau_{M,f}$ and $\tau_{N,g}$ coincide in $\Hom_{\Z\pi}(H^2(\pi;\Z\pi);\Z/2)$.
  \end{enumerate}
  \end{enumerate}
\end{thm}

Hambleton, Kreck, and the third author~\cite{HKT} previously classified $4$-manifolds with 2-dimensional fundamental groups, up to $s$-cobordism, in terms of the equivariant intersection form. They also assumed that the assembly map $H_4(B\pi;\mathbb{L}\langle 1 \rangle) \to L_4(\Z\pi)$ is injective, where $\mathbb{L}\langle 1\rangle$ denotes the 1-connected cover of the $L$-theory spectrum of the integers~\cite{Ranicki-blue-book}. This in particular holds for a torsion-free group $\pi$ whenever the Farrell--Jones conjecture holds for $\pi$ by \cite[Lemma~3.3]{hambleton-stability}.
%If $\pi$ is good this yields the homeomorphism classification.
While Theorem~\ref{thm:geom-2-dim-gps} only concerns the stable classification, it has the advantage that to apply it one only needs to compute a relatively small number of Kervaire--Milnor invariants, compared with computing the entire intersection form.

By \cref{thm:geom-2-dim-gps}, the map $\tau_{M,f}$ is a homomorphism if $\pi$ is 2-dimensional. In \cref{thm:main-technical-thm-intro}, we will see that  this is also the case whenever $f$ factors through a 2-dimensional complex. But in general the following question remains open.

\begin{question}
	Is the map $\tau_{M,f} \colon \ker \Sq \to \Z/2$ always a homomorphism?
\end{question}

\begin{example}\label{example:surface-groups}
Let $\Sigma$ be a closed, oriented surface with positive genus and suppose $\pi = \pi_1(M) \cong \pi_1(\Sigma)$. Then the radical of $\lambda_M$ is isomorphic to $H^2(\pi;\Z\pi)\cong H^2(\Sigma;\Z\pi)\cong H_0(\Sigma;\Z\pi) \cong \Z$. In this case our classification is particularly efficient since it requires the computation of just a single Kervaire--Milnor invariant $\tau(S)$, where $[S]$ generates $\Z/2\otimes_\Z \Rad(\lambda_{M}) \cong \Z/2$. In particular, $\tau_{M,f}$ is independent of the choice of $f$.

Among closed, smooth 4-manifolds with $\pi_1(M) \cong \pi_1(\Sigma)$ and signature zero, there are two stable diffeomorphism classes. The class with trivial $\tau_M$ is represented by $M=\Sigma \times S^2$ where the radical $\Rad(\lambda_M) = \pi_2(M) \cong \Z$ is generated by an embedded sphere $\{\pt\}\times S^2$. The second stable diffeomorphism class is represented by a 4-manifold $M'$ constructed from $\Sigma \times T^2$ by performing surgery on framed circles representing a dual pair of generators of $\pi_1(T^2) \cong \Z^2$, where the framing of the circles is ``twisted''.  The generator of $\Rad(\lambda_{M'}) \cong \Z$ cannot be represented by an embedding, even stably.
\end{example}

As mentioned above, \cref{thm:geom-2-dim-gps} follows from our main technical theorem, \cref{thm:main-technical-thm-intro}, which we will explain below. First we review the definition of the Kervaire--Milnor invariant and the reformulation of the stable diffeomorphism question into bordism theory by Kreck~\cite[Theorem~C]{surgeryandduality}.

\subsection{Review of the Kervaire--Milnor invariant}\label{sub-sec:review-KM}

The Kervaire--Milnor invariant appeared previously in \cite{Freedman-Quinn}, \cite{Stong}, and \cite{schneiderman-teichner-tau}, following a closely related invariant defined in~\cite{Kirby-Freedman, Matsumoto}.
A version of this invariant was used by Freedman and Quinn to detect the Kirby--Siebenmann obstruction to smoothing the topological tangent bundle of a simply-connected topological 4-manifold with odd intersection form, in particular detecting the difference between $\mathbb{CP}^2$ and its star partner $*\mathbb{CP}^2$.
This is somewhat orthogonal to the appearance of this invariant in the stable classification of spin 4-manifolds, since for spin topological 4-manifolds, it follows from Rochlin's theorem that the Kirby--Siebenmann invariant is computed as the signature divided by $8$, and then modulo two.

Let $M$ be a smooth, spin $4$-manifold with fundamental group $\pi$ and equivariant intersection form
 \begin{align*}
   \lambda_M\colon \pi_2(M) \times \pi_2(M) &\to \Z\pi\\
   (x,y) & \mapsto \langle PD^{-1}(y),x\rangle.
 \end{align*}
%is the equivariant intersection form.
To recall the definition of the Kervaire--Milnor invariant, suppose that $x \in \pi_2(M) \cong H_2(M;\Z\pi)$ satisfies $\lambda_M(x,x)=0$. Then we can represent $x$ by a generic immersion $S \colon S^2 \looparrowright M$ whose double points can be paired up by generically immersed Whitney discs $\{W_i\}$; see e.g.\ \cite[Proposition~11.10]{the-book-ch11}, or \cite{KPRT}.
By boundary twisting and pushing down~\cite[Sections~1.3~and~2.5]{Freedman-Quinn}, the Whitney discs can be chosen to be disjointly embedded, framed, and to intersect $S$ transversely.  Then the \emph{Kervaire--Milnor invariant} of $S$ is
\[\tau(S;\{W_i\}) := \sum_i |\mathring{W}_i \pitchfork S| \mod{2}.\]
Suppose that $x$ is \emph{$\RPT$-characteristic}, meaning that for every map $R \colon \RPT \to M$, \[\lambda_2(\red_2(x) , [R]) = 0,\]
where $\red_2 \colon \pi_2(M) \cong H_2(M;\Z\pi) \to H_2(M;\Z/2)$ is again the map induced by the modulo two augmentation, $\lambda_2 \colon H_2(M;\Z/2) \times H_2(M;\Z/2) \to \Z/2$ is the mod 2 intersection pairing, and $[R]$ denotes the image of the generator $[\RPT]\in H_2(\RPT;\Z/2)$ under $R_*$.
 Then $\tau(S;\{W_i\})$ is well-defined~\cite{schneiderman-teichner-tau} on the homotopy class $x \in \pi_2(M)$, independent of the choices of $S$ and the $\{W_i\}$, and so we write \[\tau(x) := \tau(S;\{W_i\}) \in \Z/2.\]
In \cref{sec:tau} we will give more details on the Kervaire--Milnor invariant, as well as relating it with an equivalent definition that is used in the proof of our theorems.

\subsection{Review of the stable classification via spin bordism}\label{subsection:review}

Kreck~\cite[Theorem~C]{surgeryandduality} showed  that two closed, spin $4$-manifolds with fundamental group $\pi$ are stably diffeomorphic if and only if there are choices of spin structures and identifications of the fundamental groups with $\pi$, giving rise to equal elements in the bordism group $\Omega_4^{\spin}(B\pi)$.  To understand this group of bordism classes of pairs $(M,c)$, where $M$ is a closed 4-manifold with spin structure and $c\colon M\to B\pi$ classifies the universal cover, we consider the Atiyah--Hirzebruch spectral sequence computing $\Omega_4^{\spin}(B\pi)$.  Since we will use this spectral sequence for spaces other than $B\pi$, we recall it in the necessary generality, for an arbitrary topological space $X$. It takes the form:
\[E_{p,q}^2 = H_p(X;\Omega_q^{\spin}) \, \Rightarrow \, \Omega_4^{\spin}(X).\]
The Atiyah--Hirzebruch spectral sequence gives rise to a filtration whose iterated graded quotients are
\[
\Z \cong \Omega_4^{\spin} \underbrace{\subseteq}_{E_{2,2}^\infty} F_{2,2} \underbrace{\subseteq}_{E_{3,1}^\infty} F_{3,1}   \underbrace{\subseteq}_{E_{4,0}^\infty} \Omega_4^{\spin}(X).
\]
The first isomorphism is determined by the signature divided by 16.
The signature extends to the entire group $\Omega_4^{\spin}(X)$ and so we reduce our study to $\widetilde\Omega_4^{\spin}(X)$, the kernel of the signature map.
The  spectral sequence then reduces to a shorter filtration
\[
E_{2,2}^\infty \underbrace{\subseteq}_{E_{3,1}^\infty} F   \underbrace{\subseteq}_{E_{4,0}^\infty} \widetilde\Omega_4^{\spin}(X),
\]
where the subgroup $F$ consists of bordism classes represented by signature zero $4$-manifolds $M$ with spin structure such that $c\colon M\to X^{(3)}$ lands in the 3-skeleton of $X$. Similarly, the smallest filtration term $E_{2,2}^\infty$ is represented by elements $(M,c)$ with $c\colon M\to X^{(2)}$.

Now we restrict to $X=B\pi$.
Since the $E^2_{p,q}$ term of the spectral sequence is $H_p(\pi;\Omega_q^{\spin})$, the $E^\infty_{p,q}$-terms in this case are as follows:
\begin{align*}
E_{2,2} &:= E^\infty_{2,2}=H_2(\pi;\Z/2)/\im (d_2, d_3);\\
E_{3,1} &:= E^\infty_{3,1}=H_3(\pi;\Z/2)/\im (d_2);\\
E_{4,0} &:= E^\infty_{4,0}= \ker (d_2,d_3) \subseteq  H_4(\pi;\Z).
\end{align*}
%\commentm{MP: For the last one, why isn't there a differential $d_3 \colon \ker (d_2\colon  H_4(\pi;\Z)\to H_2(\pi;\Z/2)) \to H_1(\pi;\Z/2)$?}
Moreover, by \cite[Theorem 3.1.3]{teichnerthesis} the $d_2$ differentials are given by the dual homomorphisms $\Sq_2 \colon H_{i+2}(\pi;\Z/2) \to H_{i}(\pi;\Z/2)$ to the Steenrod squares $\Sq^2 \colon H^i(\pi;\Z/2) \to H^{i+2}(\pi;\Z/2)$, in the case $i=3$ precomposed with the homomorphism induced by reduction modulo 2, $\red_2 \colon H_{i+2}(\pi;\Z) \to H_{i+2}(\pi;\Z/2)$.
Following \cite{teichnerthesis}, we obtain the primary invariant $\pri(M)=c_*[M] \in E_{4,0}$, the secondary invariant $\msec(M) \in E_{3,1}$ and the tertiary invariant $\ter(M) \in E_{2,2}$.

\subsection{Relating the tertiary and Kervaire--Milnor invariants}\label{subsection:relating-ter-km}

We studied the primary invariant in \cite{KPT18}, and we studied the secondary and tertiary invariants in \cite{KPT20}, building on \cite{KLPT15}.
In \cite{KPT18} and \cite{KPT20}, we gave criteria which can decide whether $(M,c) \in E_{2,2} \subseteq \widetilde\Omega_4^{\spin}(B\pi)$, that is whether $(M,c)$ is bordant to $(M',c')$ such that $c' \colon M' \to B\pi$ factors through the 2-skeleton $B\pi^{(2)} \subseteq B\pi$.
Our main theorem, stated next, says that assuming there is such a bordism, one can compute the tertiary invariant $\ter(M')$ using the Kervaire--Milnor invariant.

One can compute $\ter(M')$ via a collection of codimension two Arf invariants.  The difficulty with this in practice is that one need to first find a homotopy of our $2$-connected map to the 2-skeleton of $B\pi$, take precise inverse images of regular points, and compute with spin structures.  Since the Kervaire--Milnor invariant is well-defined on homotopy classes, and does not depend on the choice of spin structure, it represents a computational improvement.

\begin{thm}\label{thm:main-technical-thm-intro}
Let $M$ be a closed, smooth, spin $4$-manifold with fundamental group $\pi$.    Suppose that there is a map $f\colon M\to K$ to a 2-complex $K$ that is an isomorphism on fundamental groups.  Let $i \colon K \to B\pi$ be a $2$-connected map.
 \begin{enumerate}[(i)]
\item\label{it:main0} For each $\phi\in \ker (\Sq^2 \colon H^2(\pi;\Z/2) \to H^4(\pi;\Z/2))$, there exists a lift $\widehat{\phi} \in H^2(K;\Z\pi)$ of $i^*(\phi) \in H^2(K;\Z/2)$.
   \item\label{it:main1}  The element $PD(f^*(\widehat{\phi})) \in H_2(M;\Z\pi) \cong \pi_2(M)$ is $\RPT$-characteristic and has trivial self-intersection number, so that the Kervaire--Milnor invariant $\tau(PD(f^*(\widehat{\phi}))) \in \Z/2$ is well-defined.
 \item\label{it:main2} The map
 \begin{align*}
   \wh\tau_{M,f} \colon \ker \Sq^2 &\to \Z/2 \\
   \phi & \mapsto \tau_{M,f}(\wh\phi)=\tau(PD(f^*(\widehat{\phi})))
 \end{align*}
is well-defined $($i.e.\ is independent of the choice of $\widehat{\phi})$ and is a homomorphism.
\item\label{it:main3} Under the map
\[\Hom(\ker \Sq^2,\Z/2) \xrightarrow{\cong} H_2(\pi;\Z/2)/\im \Sq_2 \to H_2(\pi;\Z/2)/\im (d_2,d_3),  \]
where the first isomorphism is the inverse of the evaluation map,  $\wh\tau_{M,f}$ is sent to $\ter(M)$.
Here $d_2 = \Sq_2 \colon H_4(\pi;\Z/2) \to H_2(\pi;\Z/2)$ and $d_3 \colon H_5(\pi;\Z) \to H_2(\pi;\Z/2)/\im(d_2)$ are the differentials in the Atiyah--Hirzebruch spectral sequence as in \cref{subsection:review}.  In particular, the image of $\wh\tau_{M,f}$ under the displayed map is independent of the choices of $f$ and~$K$.
 \end{enumerate}
\end{thm}

\begin{proof}[Proof of \cref{thm:geom-2-dim-gps} assuming \cref{thm:main-technical-thm-intro}]
For $\pi$ a $2$-dimensional group, take $K=B\pi$ and $f \colon M \to B\pi$ as in \cref{thm:geom-2-dim-gps}.
%Since $B\pi$ is 2-dimensional, $d_2= \Sq_2 =0$ and $d_3=0$, and thus
Since $H^4(\pi;\Z/2)=0$, we have that $\Sq^2=0$, and thus $\wh\tau_{M,f}$ is a homomorphism $H^2(\pi;\Z/2)\to \Z/2$ by \cref{thm:main-technical-thm-intro} \eqref{it:main2}.
For $x\in H^2(\pi;\Z\pi)$, we have $\tau_{M,f}(x)=\wh\tau_{M,f}([x])$, where $[x]\in H^2(\pi;\Z/2)$ is the image of $x$ under mod 2 augmentation. Hence $\tau_{M,f}\colon H^2(\pi;\Z\pi)\to \Z/2$ is composition of two homomorphisms, so is a homomorphism. This shows \cref{thm:geom-2-dim-gps} \eqref{it:thm1.2.1}.

Since $E_{3,1}=E_{4,0}=0$, $M$ and $N$ are stably diffeomorphic if and only if their signatures are equal and $\ter(M)=\ter(N)\in H_2(\pi;\Z/2)$ by \cite[Theorem~C]{surgeryandduality}.
Again  $B\pi$ is 2-dimensional, $d_2= \Sq_2 =0$ and $d_3=0$, and hence the displayed map in  \cref{thm:main-technical-thm-intro} \eqref{it:main3} is an isomorphism $\Hom(H^2(\pi;\Z/2),\Z/2) \to H_2(\pi;\Z/2)$ that sends $\wh\tau_{M,f}$ to $\ter(M)$, and similarly for $N$. Thus $\ter(M)=\ter(N)$ if and only if $\wh\tau_{M,f}=\wh\tau_{N,g}$. Since by definition, $\wh\tau_{M,f}([x])=\tau_{M,f}(x)$ and $H^2(\pi;\Z\pi)\to H^2(\pi;\Z/2)$ is surjective, it follows that $\tau_{M,f}=\tau_{N,g}$ if and only if $\wh\tau_{M,f}=\wh\tau_{N,g}$. This completes the proof of \cref{thm:geom-2-dim-gps}.
%
%
%%%%%%%%%%%%%%
% For $\pi$ a $2$-dimensional group, take $K=B\pi$, $f \colon M \to B\pi$ as in \cref{thm:geom-2-dim-gps}, and note that $E_{3,1} = E_{4,0} = 0$ and $\ker \Sq^2  =  H^2(\pi;\Z/2)$.   Given $x \in \Rad(\lambda_M) = PD(f^*(H^2(\pi;\Z\pi)))$ take $\widehat{\phi} \in H^2(\pi;\Z\pi)$ to be a preimage. Setting $\phi \in \ker \Sq^2$ to be the image of $\widehat{\phi}$ under reduction from $\Z\pi$ to $\Z/2$ coefficients, we have that $x= PD(f^*(\widehat{\phi}))$.  \cref{thm:geom-2-dim-gps} now follows from \cref{thm:main-technical-thm-intro}~\eqref{it:main1}, \eqref{it:main2}, and \eqref{it:main3}, our analysis of the Atiyah-Hirzebruch spectral sequence above, and \cite[Theorem~C]{surgeryandduality}.
\end{proof}

\subsubsection*{Organisation of the paper}

In \cref{sec:tau} we provide more details on the Kervaire--Milnor invariant of immersed spheres, and give an alternative equivalent description in terms of $\pi_1$-trivial immersed surfaces.
In \cref{section:Arf-and-AHSS} we explain how the Arf invariant arises in the Atiyah--Hirzebruch spectral sequence computation of spin bordism.
In \cref{sec:tertiary-obstruction} we show \cref{thm:main-technical-thm-intro} by comparing the Kervaire--Milnor invariant with the Arf invariant.

\subsubsection*{Acknowledgments}

We are grateful to the Max Planck Institute for Mathematics in Bonn, and to Rob Schneiderman and Arunima Ray for numerous useful conversations on the Kervaire--Milnor invariant. We also thank an anonymous referee for a detailed and thoughtful report.  DK was partially funded by the Deutsche Forschungsgemeinschaft under Germany's Excellence Strategy -- GZ 2047/1, Projekt-ID 390685813.  MP was partially supported by EPSRC New Investigator grant EP/T028335/2 and EPSRC New Horizons grant EP/V04821X/2.

%For the purposes of open access, the authors have applied a CC-BY license to the author-accepted manuscript arising from this submission.

\section{The Kervaire--Milnor invariant}\label{sec:tau}
Let $M$ be a smooth, closed, oriented, spin 4-manifold.
In this section, a \emph{surface in $M$} is an abstract surface $\Sigma$ together with a generic immersion $F \colon \Sigma \to M$, meaning the map is an immersion, and all intersections and self-intersections are transverse double points. In particular there are no triple points. Moreover the boundary of $\Sigma$, if nonempty, is assumed to be embedded in $M$. We will typically denote the data $(\Sigma,F)$ just by $F$ for brevity.

%Let $M$ be a smooth, closed, oriented, spin 4-manifold.
%In this section, all surfaces are assumed to be the images of generic maps into $M$, meaning the maps are immersions, and all intersections and self-intersections are transverse double points. In particular there are no triple points. Moreover the boundary, if nonempty, is assumed to be embedded.

Let $S$ be a generically immersed sphere in $M$ with $\lambda_M([S],[S])=0$.
Generalising  Freedman--Kirby~\cite[p.~93]{Kirby-Freedman}, Guillou--Marin~\cite{Guillou-Marin}, Matsumoto~\cite{Matsumoto}, and Freedman--Quinn~\cite[Definition~10.8]{Freedman-Quinn}, work of Schneiderman and the third-named author~\cite{schneiderman-teichner-tau} defined an invariant $\wt{\tau}(S)$ with values in a quotient of $\Z[\pi\times\pi]$.
Which quotients of $\Z[\pi \times \pi]$ one can take in order to get an invariant of the homotopy class of $\Sigma$ depends on the intersection numbers of $\Sigma$ with other immersed surfaces in~$M$.

Assuming that $S$ is $\RPT$-characteristic, the image of $\wt{\tau}(S)$  under the augmentation and reduction modulo two map $\Z[\pi \times \pi] \to \Z/2$ is a well-defined invariant of the homotopy class of $\Sigma$.
Following the nomenclature of Freedman--Quinn, we call this image the Kervaire--Milnor invariant $\tau(S) \in \Z/2$ of $[S] \in \pi_2(M)$.

We will define $\tau(S)$ carefully in \cref{section:tau-spheres}. Then in \cref{section:tau-surfaces} we will extend the definition to $\pi_1$-trivial generically immersed closed, oriented surfaces in $M$.

\subsection{The \texorpdfstring{$\tau$}{tau} invariant for generically immersed spheres}\label{section:tau-spheres}

As before, let $M$ be a smooth, closed, oriented, spin 4-manifold, and fix an identification $\pi_1(M) \cong \pi$.

\begin{defi}[Self-intersection number {\cite[Chapter~5]{Wall}, \cite[Section~1.7]{Freedman-Quinn}}]
\label{defn:mu}
	Let $x \in \pi_2(M)$.  Since $M$ is spin, we can represent $x$ by a generically immersed sphere $S' \colon S^2 \to M$ whose normal bundle has even Euler number.  Add cusp homotopies (see e.g.\ \cite[Chapter~XII, p.~72]{Kirby-4-manifold-book} for the local model) in a small open set to make the Euler number of the normal bundle zero, and  call the resulting sphere $S$.  It can be checked in the local model that a cusp homotopy changes the Euler number of the normal bundle by $\pm 2$.

 Now count the self intersections of $S$ with sign and group elements.  The attribution of signs uses the orientation of~$M$.  The group element is the image in~$\pi_1(M)$ of a double point loop associated to the self-intersection point, with some choice of orientation of the double point loop.  This count gives rise to an element
 \[\mu(x) \in \Z\pi/\{g \sim g^{-1}\mid g\in \pi\}.\]
 This \emph{self-intersection number} is valued in a quotient abelian group of the $\Z\pi$-module $\Z\pi$. The indeterminacy arises because there is no canonical way to decide whether to associate $g$ or $g^{-1}$ to a given  double point of $\Sigma$.  The number $\mu(x)$ is an invariant of the homotopy class of $x$.
\end{defi}

\begin{remark}
The normalisation of $\mu(x)$ at $1\in \pi$, obtained by choosing the regular homotopy class whose normal bundle has trivial Euler number, implies that $\lambda_M(x,x)=\mu(x)+\ol{\mu(x)}\in \Z\pi$, where $\ol{(-)}\colon \Z\pi\to \Z\pi$ is determined by $g\mapsto g^{-1}$ and in order to see that the right hand side is well-defined in $\Z\pi$ we use that $\gamma \mapsto \gamma +\ol{\gamma}$ factors through $\Z\pi/\{g\sim g^{-1}\}$.
This normalisation works if the Euler number $e(\nu S)$ is even for every $S$, or equivalently if $w_2(\widetilde M)=0$, i.e.\ the universal cover of $M$ is spin.

On the other hand, using cusp homotopies it is always possible to change $S$ so that
the self-intersection number of $S$ at $1$ is
trivial, even if the universal cover of $M$ is not spin. This gives an element  $\mu'(x)\in \Z\pi/\{g\sim g^{-1}\}$ which
again only depends on the homotopy class of $x$. Using this convention, $\mu'(x)$ is an obstruction to representing $x$ by an embedded sphere.  In the setting of this paper, our 4-manifolds are spin and we usually assume that $\lambda_{M}(x,x)=0$. In that case, the two conventions agree and $\mu(x)=\mu'(x)=0$.
\end{remark}

The following lemma is rather useful, since that it tells us that it is enough to consider the equivariant intersection pairing in order to find spheres with vanishing self-intersection number.

\begin{lemma}\label{lem:lambda-zero-mu-zero}
For a closed, oriented, spin 4-manifold $M$, if $\lambda_M(x,x)=0$ for $x\in \pi_2(M)$, then $\mu(x)=0$.
\end{lemma}

\begin{proof}
Using a representative as in \cref{defn:mu}, we can assume that $0= \lambda_M(x,x) = \mu(x) + \ol{\mu(x)}$.
	Suppose that $\sum_g n_g g \in \Z\pi$ is a lift of $\mu(x)$.  Then $\mu(x) + \ol{\mu(x)}=0$ implies that $n_g + n_{g^{-1}} = 0$ for every $g \in \pi$. If $g= g^{-1}$ then we immediately see $n_g=0$.  For the remaining group elements, in the value group of $\mu$ we have $g \sim g^{-1}$, so $n_g g +n_{g^{-1}} g^{-1} = (n_g + n_{g^{-1}})g = 0\cdot g =0$.  Sum over a set of representatives for the subsets $\{g,g^{-1}\}$ with $g \neq g^{-1}$, to obtain $\mu(x)=0$.
\end{proof}

%Denote the map given by augmentation composed with reduction modulo 2 by $\red_2\colon \Z\pi \xrightarrow{\epsilon} \Z \to \Z/2$.
%We abuse notation and also use $\red_2 \colon \pi_2(M) = H_2(M;\Z\pi) \to H_2(M;\Z/2)$ to denote the induced map on homology.

\begin{definition}\label{defn:spherically-characteristic}
Let $\lambda_2 \colon H_2(M;\Z/2) \times H_2(M;\Z/2) \to \Z/2$ be the $\Z/2$-valued intersection pairing.	
\begin{enumerate}
		\item
		An element $\alpha\in\pi_2(M)$ is called  \emph{$S^2$-characteristic} if $\red_2(\lambda_M(\alpha,\beta)) = 0 \in \Z/2$ for all $\beta \in \pi_2(M)$.  Let $\SC\subseteq \pi_2(M)$ denote the subset of $S^2$-characteristic elements $\alpha$ with $\mu(\alpha)=0$.
		\item An element $\alpha\in\pi_2(M)$ is called  \emph{$\RPT$-characteristic} if $\lambda_2(\red_2(\alpha),[R])) = 0 \in \Z/2$ for every map $R \colon \RPT \to M$.  Let $\RC\subseteq \pi_2(M)$ denote the subset of $\RPT$-characteristic elements $\alpha$ with $\mu(\alpha)=0$.
	\end{enumerate}
\end{definition}

\begin{lemma}\label{lem:R-char-implies-s-char}
	An $\RPT$-characteristic sphere $\alpha \in \pi_2(M)$ is $S^2$-characteristic.
	Moreover if $\pi_1(M)$ has no elements of order two, then $\alpha$ is $S^2$-characteristic if and only if it is $\RPT$-characteristic.
\end{lemma}

\begin{proof}
	A generic immersion $f \colon S^2 \imra M$ determines a map $\RPT \to \RPT/\mathbb{RP}^1 = S^2 \xrightarrow{f} M$, which can be perturbed to a generic immersion of $\RPT$ into $M$ with the same intersection behaviour with $\alpha$ as the original $S^2$. Thus $\RPT$-characteristic implies $S^2$-characteristic.
	
	On the other hand, if no element of $\pi_1(M)$ has order 2, then for every generic immersion $R$ of $\RPT$, the induced map $\pi_1(\RPT) \to \pi_1(M)$ is the zero map. Therefore $R$ is homotopic to a map that factors as $\RPT \to \RPT/\mathbb{RP}^1 = S^2 \xrightarrow{f} M$, and intersections with $f(S^2)$ agree with intersections with $R$. It follows that $S^2$-characteristic implies $\RPT$-characteristic.
\end{proof}

Let $S\colon S^2 \looparrowright M$ be a generically immersed 2-sphere with vanishing self-intersection number $\mu(S)=0$.  Then the self-intersection points of $S$ can be paired up so that each pair consists of two points having oppositely signed but equal group elements associated to their double point loops.  Therefore, one can choose a Whitney disc $W_i$ for each pair of self-intersections, and arrange that all the boundary arcs are disjoint.  The normal bundle to the disc $W_i$ has a unique framing, and the Whitney framing of the normal bundle of $W_i$ restricted to $\partial W_i$ differs from the restriction of the disc framing by an integer $n_i \in \Z$.  (The Whitney framing is determined by a section of the normal bundle $\nu_{W_i}|_{\partial W_i}$ that lies in $DS(TS^2) \cap \nu_{W_i}$ along one boundary arc of $\partial W_i$ and lies in $\nu_{S^2} \cap \nu_{W_i}$ along the other boundary arc.)

\begin{definition}
If $S$ is $\RPT$-characteristic, then
\[
\tau(S) := \sum_i |\mathring{W}_i \pitchfork S| + n_i \mod{2}.
\]
\end{definition}

\begin{lemma}
  The expression $\tau(S)$ is independent of the choice of pairings of double points, sheet choices and Whitney arcs, and Whitney discs.  Moreover, $\tau(S)$ only depends on the regular homotopy class of the generic immersion.
\end{lemma}

\begin{proof}
The lemma essentially follows from \cite[Theorem~1]{schneiderman-teichner-tau}.
% See \cite[Theorem~1]{schneiderman-teichner-tau} for a proof that this number is well-defined.
A key observation here is due to Stong~\cite{Stong}.  We make a couple of remarks on how to translate the version in \cite{schneiderman-teichner-tau} to the current version.  First note that in the formulation of \cite{schneiderman-teichner-tau}, as mentioned above the intersections were decorated with a pair of fundamental group elements, to give an invariant in a quotient of $\Z[\pi \times \pi]$ by certain relations. Since we consider the augmentation followed by the reduction modulo two, all but the last relation given in \cite[Theorem~1]{schneiderman-teichner-tau} are vacuous.  In addition their last relation is irrelevant because we consider $\RPT$-characteristic elements.  Secondly, the formulation of Schneidermann--Teichner requires that Whitney discs be framed, whereas we do not, and include the framing coefficient as part of the definition.  However by boundary twisting \cite[Section~1.3]{Freedman-Quinn}, one can alter $n_i$ to be zero at the cost of introducing $|n_i|$ intersection points in $\mathring{W}_i \pitchfork S$, and so the two ways of computing $\tau(S)$ agree.
\end{proof}

For every element of $\pi_2(M)$, we fix a regular homotopy class within the homotopy class by the requirement that the Euler number of the normal bundle be zero. That fixing the Euler number determines a regular homotopy class of immersions is a consequence of Smale--Hirsch immersion theory; for  details specific to surfaces in 4-manifolds, see for example \cite[Theorem~2.32~(3)]{KPRT}.   Thus $\tau$ becomes well-defined on $\RC \subseteq \pi_2(M)$. So we have defined a map \[\tau \colon \RC \to \Z/2.\]

\begin{remark}
	If $S$ is not $S^2$-characteristic then $\tau(S)$ is not well-defined, since adding a sphere that intersects $S$ in an odd number of points to one of the Whitney discs would change the sum in the definition of $\tau$ by one.
	
	If $S$ is $S^2$-characteristic but not $\RPT$-characteristic, then $\tau(S)$ is also not well-defined, as observed by Stong~\cite{Stong}. In this case,
 %by \cref{lem:R-char-implies-s-char} there is $2$-torsion in $\pi_1(M)$. It is then possible,
  a change in choice of Whitney arcs, can also change~$\tau(S)$.
\end{remark}

\subsection{The \texorpdfstring{$\tau$}{tau} invariant for \texorpdfstring{$\pi_1$}{pi1}-trivial generically immersed surfaces}\label{section:tau-surfaces}

In this subsection we introduce the following extension of the $\tau$ invariant, which is defined on $\RPT$-characteristic, $\pi_1$-trivial, generically immersed surfaces  instead of on $\RPT$-characteristic generically immersed spheres.
We will not need the full version of this invariant, only the embedded version.  But we anticipate that the full version might be useful in the future, so we include it here, as it requires little extra work.

As before let $M$ be a smooth, closed, oriented, spin 4-manifold together with an identification $\pi_1(M) \cong \pi$.
We call a generically immersed, closed, oriented surface $F \colon \Sigma \looparrowright M$ a \emph{$\pi_1$-trivial surface} if $F_* \colon \pi_1(\Sigma) \to \pi_1(M)$ is the trivial map.

\begin{definition}
A $\pi_1$-trivial generically immersed surface $F\colon \Sigma\looparrowright M$ is said to be \emph{$\RPT$-characteristic} if it intersects every generically immersed $\RPT$ in general position in an even number of points i.e.\ if the element of $\pi_2(M)$ determined by $F$ via the Hurewicz isomorphism $H_2(M;\Z\pi) \cong \pi_2(M)$ is $\RPT$-characteristic.
\end{definition}

A $\pi_1$-trivial $\RPT$-characteristic generically immersed surface $F$ has a self-intersection number $\mu(F) \in \Z\pi/\{g \sim g^{-1} \mid g \in \pi\}$ defined as follows.
%Add local kinks to $F$ until
Change $F$ by cusp homotopies such that
its normal bundle is trivial; this is possible since $F$ is $S^2$-characteristic by \cref{lem:R-char-implies-s-char}.
Now count self-intersection points of the generically immersed surface with group elements and sign.  We use $\pi_1$-triviality to see that the associated group elements do not depend on the choice of double point loop on $F$ used to compute it.
\medskip

Let $F \colon \Sigma\looparrowright M$ be a generically immersed $\pi_1$-trivial surface with $\mu(F)=0$, and let $\alpha$ be an embedded circle in $\Sigma$ such that the restriction of $F$ to $\alpha$ is an embedding. The circle $F(\alpha)$ bounds a disc $C$ in $M$, since $F$ is $\pi_1$-trivial. The normal direction of $\alpha$ in $\Sigma$, pushed forward into $M$, gives a section of the normal bundle of $C$ at the boundary $F(\alpha)$. Therefore, the relative Euler number $e(C)$ of the normal bundle of $C$ is a well-defined integer. We define
\[\varpi(\alpha):=|\mathring{C} \pitchfork F|+e(C) \mod{2},\]
where $|\mathring{C}\pitchfork F|$ is the number of transverse intersections between the interior of $C$ and $F(\Sigma)$.

\begin{lemma}\label{lemma:independence-of-C}
  If $F$ is $S^2$-characteristic, then the count $\varpi(\alpha)$ does not depend on the choice of $C$.
\end{lemma}

\begin{proof}
  Let $C$ and $C'$ be two choices of discs with boundary $F(\alpha)$, and let $\varpi_C(\alpha)$ and $\varpi_{C'}(\alpha)$ temporarily denote the count made using $C$ and $C'$ respectively.  Perform boundary twists~\cite[Section~1.3]{Freedman-Quinn} in order to arrange that $C$ and $C'$ are framed with respect to their boundaries, i.e.\ $e(C)=e(C')=0$. Boundary twists do not change the counts~$\varpi_C(\alpha)$ and $\varpi_{C'}(\alpha)$, since a boundary twist changes the relative Euler number of the disc by one and produces a single new intersection between the disc being twisted and $F$.  Now rotate $C'$ near $F(\alpha)$ so that the union of $C$ and $C'$ is a generically immersed 2-sphere. Since $F$ is $S^2$-characteristic, we have
  %\[0 = \red_2\lambda(F,C \cup_{\alpha} C') = \varpi_C(\alpha) + \varpi_{C'}(\alpha) \in \Z/2\]
  \[0 = \lambda_2([F],[C \cup_{\alpha} C']) = \varpi_C(\alpha) + \varpi_{C'}(\alpha) \in \Z/2\]
  as desired.
\end{proof}

Consider a hyperbolic basis of $H_1(\Sigma;\Z)$ represented by embedded circles $a_1,\ldots, a_n, b_1,\ldots, b_n$ that are disjoint from each other except that $a_i$ intersects $b_i$ transversely in a single point. Suppose that the restriction of $F$ to each of the $a_i$ and the $b_j$ is an embedding.

Since $\mu(F)=0$, all double points of $F$ can be paired up by generically immersed Whitney discs $W_1, \ldots , W_m \hookrightarrow M$ whose boundary arcs on $F(\Sigma)$ are disjoint from each other, the $F(a_i)$, and the $F(b_i)$. Let $n_j$ again denote the framing coefficient of the Whitney discs discussed in \cref{section:tau-spheres}.
Then define:
\[\tau(F):=\sum_{i=1}^n\varpi(a_i)\varpi(b_i) + \sum_{j=1}^m |\mathring{W}_j\pitchfork F|+n_j \mod 2.\]

\begin{remark}
Note that in the case that $\Sigma$ has genus zero, this reduces to the $\tau$ invariant of the previous subsection since the first sum vanishes.  Also note that in the case of an embedded surface, only the first summand appears, and again the definition simplifies.
The restriction to the case that~$F$ is an embedding is similar to the version of $\tau$ from \cite{Kirby-Freedman}. In this case it is the Arf invariant of the quadratic form given by $\varpi$.
\end{remark}

Next we will show that $\tau(F)$ is independent of the choice of basis $\{a_i,b_i\}$, as well as the choice of Whitney discs $W_j$.

\begin{lemma}
	If $F$ is $\RPT$-characteristic, then the expression $\tau(F) \in \Z/2$ is independent of the choices of $a_i$, $b_i$, $C$ and $W_j$ made in its definition.  Moreover, $\tau(F)$ only depends on the regular homotopy class of the generic immersion.
\end{lemma}

\begin{proof}
We already showed in \cref{lemma:independence-of-C} that $\tau(F)$ is independent of the choices of the discs~$C$.

Choose a path in $M$ from each component of $F(\Sigma)$ to the base point of $M$. Since these paths are $1$-dimensional we can choose them so that the interiors of the paths do not intersect~$F$.
	Since~$F$ is $\pi_1$-trivial, it lifts to a generically immersed surface in $\wt{M}$, and hence defines an element of $H_2(\wt{M};\Z)\cong \pi_2(M)$. The strategy is to relate $\tau(F)$ to $\tau(S)$ for $S \in \pi_2(M)$, and use that $\tau(S)$ is well-defined by~\cite{schneiderman-teichner-tau}.
	
Choose generic null-homotopies $C_i\colon D^2\to M$ for $F(a_i)$ and $C_i'\colon D^2\to M$ for $F(b_i)$. As in the proof of \cref{lemma:independence-of-C}, perform boundary twists to arrange $e(C_i)=e(C_i')=0$. Again, this does not change~$\varpi(\alpha)$.
%Boundary twists do not change $\varpi(a_i)$ and $\varpi(b_i)$, since a boundary twist changes the relative Euler number of the disc by one and produces a single new intersection between the disc being twisted and $F$.
We can turn $F$ into a generically immersed $2$-sphere $S$ by performing surgeries along all the $F(a_i)$, and gluing in two parallel copies of each of the $C_i$ in place of a neighbourhood $\nu F(a_i)$ of~$F(a_i)$. We make the following observations.
\begin{enumerate}[(i)]
  \item\label{item-summand-i} Each intersection $\mathring{C}_i\pitchfork F$ yields a pair of cancelling self-intersections of $S$ paired by a Whitney disc constructed from (a parallel copy of) $C_i'$ union a band.  A schematic is shown in \cref{figure:tau-surfaces-one}.
	\begin{figure}[ht]
		\begin{center}
			\begin{tikzpicture}
				\node[anchor=south west,inner sep=0] at (0,0){\includegraphics[scale=0.3]{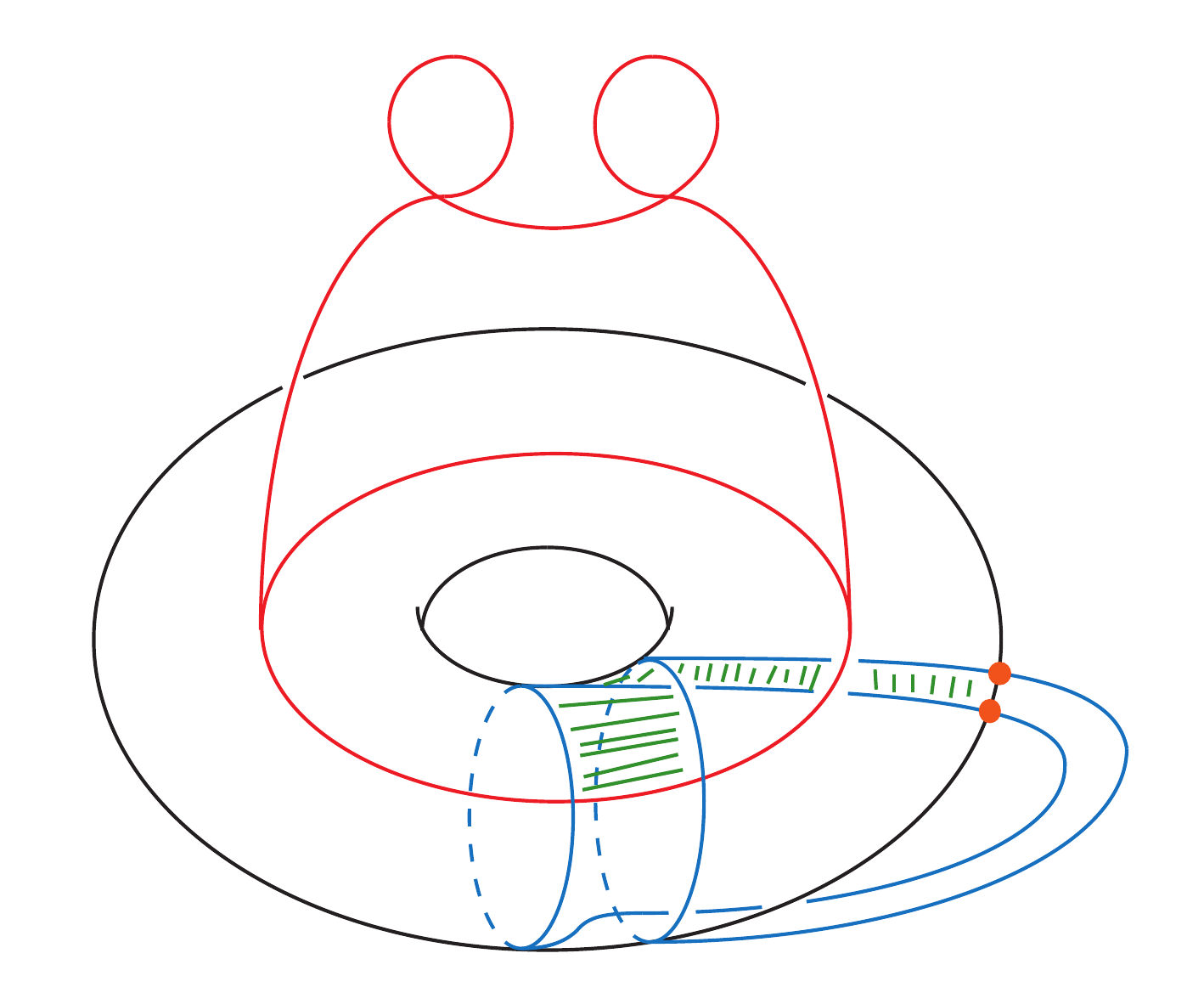}};
				\node at (6.4,2.1)  {$C$};
				\node at (2.05,5.3)  {$C'$};
			\end{tikzpicture}
		\end{center}
		\caption{A schematic of a genus one surface $F$ in $M$ with a cap $C'$ attached to the longitude, two parallel copies of a cap $C$ attached to the meridian, each of which intersect $F$ in a single point. A band is shown that, together with the cap $C'$, forms a Whitney disc pairing the two self-intersection points of the sphere obtained from surgery on $F$ using $C$.}
		\label{figure:tau-surfaces-one}
	\end{figure}
\item\label{item-summand-ii} 	Each self-intersection of $C_i$ yields two pairs of cancelling self-intersections of $S$, each with generically immersed Whitney disc a parallel copy of $C_i'$, union a band. A schematic is shown in \cref{figure:tau-surfaces-two}.
\end{enumerate}

	\begin{figure}[ht]
		\begin{center}
			\begin{tikzpicture}
				\node[anchor=south west,inner sep=0] at (0,0){\includegraphics[scale=0.3]{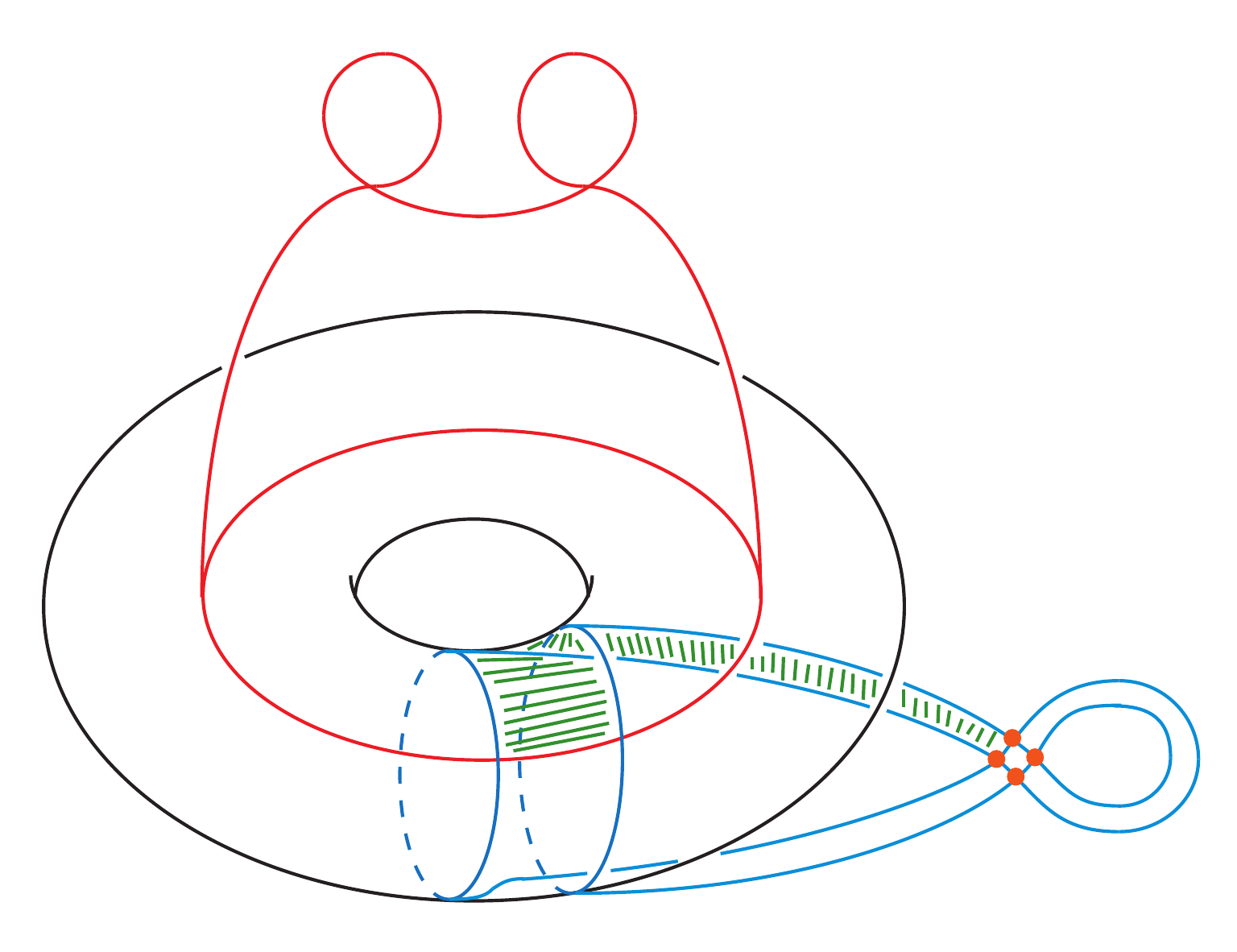}};
	\node at (6.1,1.85)  {$C$};
				\node at (1.75,5.3)  {$C'$};				
%
               %\node at (8.2,2.3)  {$C$};
			%	\node at (2.4,7)  {$C'$};
			\end{tikzpicture}
		\end{center}
		\caption{A schematic of a genus one surface $F$ with a cap $C'$ attached to the longitude and two parallel copies of a cap $C$ attached to the meridian. The cap $C$ has a single self-intersection points, which gives rise to four self-intersection points of the sphere resulting from surgery on $F$ using $C$. For one pair of these four points, a band is shown, that together with the cap $C'$, forms a Whitney disc pairing these two self-intersection points.}
		\label{figure:tau-surfaces-two}
	\end{figure}
	
	The boundary arcs of the new Whitney discs are disjoint from the boundary arcs of the old Whitney discs. Thus modulo two we see that
	\begin{align*}
		\tau(S)&=\sum_{i=1}^n\left(|\mathring{C}_i\pitchfork F|\cdot |\mathring{C}_i'\pitchfork F| + 2|\mathring{C}_i\pitchfork \mathring{C}_i|\cdot |\mathring{C}_i'\pitchfork F|\right)+ \sum_{j=1}^m |\mathring{W}_j\pitchfork F|+n_j\\&=\sum_{i=1}^n\varpi(a_i)\varpi(b_i) + \sum_{j=1}^m |\mathring{W}_j\pitchfork F|+n_j = \tau(F).
	\end{align*}
The first summand of the first summand corresponds to \eqref{item-summand-i} and the second summand corresponds to~\eqref{item-summand-ii}.

Note that $S$ and $F$ determine the same element of $\pi_2(M)$, with the right choice of basing paths,
%A choice of a collection of generic discs is needed to surger~$F$ to a sphere, but the choice of discs does not affect the homotopy class of the resulting sphere,
since $S$ and $F$ determine the same element of $H_2(\wt M;\Z)$, which in turn holds because the difference $[S] - [F]$ (viewed as singular chains instead of homology classes) bounds the trace of the surgery along the $F(a_i)$.  It follows that~$S$ is $\RPT$-characteristic. We know that the number $\tau(S)$ only depends on the homotopy class of $S$ by~\cite{schneiderman-teichner-tau}. But the homotopy class of $S$ is determined by the regular homotopy class of the generic immersion $F$ and does not depend on the choices of $a_i, b_i, C_i, C_i', W_j$. Hence the fact that $\tau(S)$ is well-defined implies that $\tau(F)$ is too.
\end{proof}

\section{The Arf invariant in the stable classification}\label{section:Arf-and-AHSS}
We will prove \cref{thm:main-technical-thm-intro} by comparing the Kervaire--Milnor invariant to a codimension two Arf invariant that arises in the Atiyah--Hirzebruch spectral sequence for $\Omega_4^{\spin}(B\pi)$.  Let us explain how this Arf invariant appears.

Let $M$ be a closed, smooth, oriented, spin $4$-manifold, where the spin structure will be fixed from now on.
Let $K$ be a $2$-complex with fundamental group $\pi$ and let $i\colon K\to B\pi$ be the $2$-connected map.
Let $f\colon M\to K$ be a map that is an isomorphism on fundamental groups.

Denote the barycentres of the $2$-cells $\{e_i^2\}_{i\in I}$ of $K$ by $\{b_i^2\}_{i \in I}$. Denote the regular preimage of the barycentre $b_i^2 \in K$ by $N^f_i \subseteq M$. We can consider $[N^f_i] \in \Omega_{2}^{\spin}$ since the normal bundle of $N^f_i$ in $M$ is trivialised as a pullback of the normal bundle of $b_i^2$ in $e_i^2$, and hence $N^f_i$ inherits a spin structure from $M$. The next lemma is well-known; see \cite[Lemma~2.5]{KLPT15} for a proof.

\begin{lemma}\label{lem:arf}
	The homomorphism $\Omega_4^{\spin}(K) \to H_2(K;\Omega_{2}^{\spin})$ from the Atiyah--Hirzebruch spectral sequence coincides with composition of the map
	\[\begin{array}{rcl} \Omega_4^{\spin}(K) &\to&  H_2^{\operatorname{cell}}(K;\Omega_{2}^{\spin}) \\
		\lbrack f \colon M \to K \rbrack  &\mapsto&  \Big[\sum\limits_{i \in I} [N^f_i]\cdot e_i^2\Big],\end{array}\]
  and the canonical isomorphism $H_2^{\operatorname{cell}}(K;\Omega_{2}^{\spin}) \cong H_2(K;\Omega_{2}^{\spin})$.
 % where $[N^f_i]\cdot e_i^2\in C_2^{\operatorname{cell}}(K;\Omega_2^{\spin})$, and we then apply the isomorphism to singular homology.
	Moreover this maps to $\ter(M)$ under
	\[H_2(K;\Omega_2^{\spin}) \cong H_2(K;\Z/2) \to H_2(\pi;\Z/2) \to H_2(\pi;\Z/2)/\im (d_2, d_3).\]
\end{lemma}

\begin{remark}\label{remark:edge-HM-derivation}
The homomorphism in the statement of the lemma $\Omega_4^{\spin}(K) \to H_2(K;\Omega_{2}^{\spin})$ arises as follows.  The abutment of the Atiyah--Hirzebruch spectral sequence $\Omega_4^{\spin}(K) = F_{4,0}$ maps to its quotient by the first filtration step $F_{0,4}=F_{1,3}$ that differs from $F_{4,0}$.  This term is indeed $F_{1,3}$, since the homology of $K$ vanishes in degrees greater than $2$, thus $E^2_{p,q} = E^{\infty}_{p,q} =0$ for all $p>2$.  Moreover, because $K$ is a 2-complex no differentials have image in $E^k_{2,2}$, for any $k$, so $E^{\infty}_{2,2} \subseteq E^{2}_{2,2}$. The composition
\[\Omega_4^{\spin}(K) = F_{4,0} \to F_{4,0}/F_{1,3} \xrightarrow{\cong} E^{\infty}_{2,2} \to E^{2}_{2,2} = H_2(K;\Omega_{2}^{\spin}).\]
gives the desired map.
\end{remark}

We will need a slight variation of \cref{lem:arf}. For $\phi\in H^2(K;\Z/2)$, represent $\phi$ by a map $K\to S^2\subseteq K(\Z/2,2)$ and let $F^f_\phi\subseteq M$ be a regular preimage of a point $s\in S^2$ under $\phi\circ f \colon M \to S^2$. As before, a framing of the normal bundle of $s$ in $S^2$ induces a framing of the normal bundle of $F^f_\phi$ in~$M$, and since~$M$ is spin, we obtain a spin structure on~$F^f_\phi$. Thus we can again consider $[F^f_\phi]$ in $\Omega_2^\spin$.

\begin{lemma}\label{lem:arf-homo-plus-computes-ter}
	The composition
	\[\Omega_4^\spin(K)\to H_2(K;\Z/2)\xrightarrow{x \mapsto \langle-,x\rangle}\Hom_{\Z/2}(H^2(K;\Z/2),\Z/2)\]
	maps $[f\colon M\to K]$ to $\big(\phi\mapsto [F^f_\phi]\in\Omega_2^\spin\cong\Z/2$\big).
\end{lemma}

\begin{proof}
By a homotopy of $f$ we can assume that the $b^2_i$ are regular points. The first map sends
\begin{align*}
\Omega_4^\spin(K) &\to H_2(K;\Z/2)\\
[f \colon M \to K] & \mapsto \sum\limits_{i \in I} [N^f_i]\cdot e_i^2
\end{align*}
 with $N_i^f := f^{-1}(b^2_i)$ as in \cref{lem:arf}.
Let $p\in S^2$ be a basepoint and let $s \in S^2$ be antipodal to $p$.
As above, given $\phi\in H^2(K;\Z/2)$, we represent $\phi$ by a map $K\to S^2\subseteq K(\Z/2,2)$.
We can choose the representative $\phi \colon K \to S^2$ so that for each 2-cell $e_i^2$ of $K$, either: (i) $\varphi|_{\ol{e}_i^2}$ sends the whole closed 2-cell to $p$, or (ii) $\varphi|_{\ol{e}_i^2}$  factors as the quotient map followed by a homeomorphism $\ol{e}_i^2 \to \ol{e}_i^2/\partial \ol{e}_i^2 \cong S^2$, with $\phi (\partial \ol{e}_i^2 ) = p$,  $\phi(b_i^2)= s$, and such that $b_i^2 \in e_i^2$ is a regular preimage of $s\in S^2$. Let $E(\phi) \subseteq I$ be the subset of indices corresponding to the cells for which the latter option (ii) holds. Then $\sqcup_{i\in E(\phi)} b_i^2$ is a regular preimage of $s\in S^2$ under $\phi$.

Let $F_\phi^f := (\phi \circ f)^{-1}(\{s\})$, as above.   As usual, $[F^f_\phi]$ does not depend on the choice of a representative for $\phi$ since different choices give spin bordant surfaces.
Hence  $F^f_\phi=\sqcup_{i\in E(\phi)} N_i^f$.
Then \[\big(\phi \mapsto [F^f_\phi]=\sum_{i\in E(\phi)}[N_i^f] \in \Omega_2^\spin\cong \Z/2\big) \in \Hom_{\Z/2}(H^2(K;\Z/2),\Z/2)\] is the image of $\sum\limits_{i \in I} [N^f_i]\cdot e_i^2$ under the evaluation map $H_2(K;\Z/2)\xrightarrow{}\Hom_{\Z/2}(H^2(K;\Z/2),\Z/2)$, as needed.
\end{proof}

For our comparison of the Kervaire--Milnor invariant with the codimension 2 Arf invariant, we need to recall the definition of the Arf invariant $\Arf\colon \Omega_2^\spin\xrightarrow{\cong}\Z/2$. Let $\Sigma$ be spin surface.
One defines a quadratic refinement of the $\Z/2$ intersection form of $\Sigma$, $\Upsilon\colon H_1(\Sigma;\Z/2)\to \Z/2$ as follows.   Represent $[\alpha] \in H_1(\Sigma;\Z/2)$ by a simple closed curve $\alpha$ in $\Sigma$. Since the normal bundle $\nu_\alpha^\Sigma$ of $\alpha$ in $\Sigma$ is one dimensional, the normal bundle $\nu_\alpha^\Sigma$ has a canonical framing, where the choice of the direction comes from the orientations. Therefore, together with the spin structure on $\Sigma$, this determines a spin structure on $\alpha$, so we may consider it as an element of $\Omega_1^{\spin}$. We define $\Upsilon([\alpha])=0$ if and only if $\alpha$ is spin null-bordant. Then $\Arf(\Sigma)$ is defined to be the Arf invariant of the quadratic form $(H_1(\Sigma;\Z/2),\lambda_\Sigma, \Upsilon)$.

\section{Proof of \texorpdfstring{\cref{thm1,thm:main-technical-thm-intro}}{Theorems 1.1 and 1.5}}
\label{sec:tertiary-obstruction}

%We recall the statement of  for the convenience of the reader.

First let us recall the setup of \cref{thm:main-technical-thm-intro}.
Let $M$ be a closed, smooth, oriented, spin $4$-manifold with fundamental group $\pi$, and suppose that there is a map $f\colon M\to K$ to a finite 2-complex $K$ that is an isomorphism on fundamental groups.  Let $i \colon K \to B\pi$ be a $2$-connected map.

We start by proving \cref{thm:main-technical-thm-intro}~\eqref{it:main0}; it follows immediately from the next lemma.
%Let $\red_2 \colon \Z\pi \to \Z/2$ be the $\Z\pi$-module homomorphism given by augmentation followed by reduction modulo $2$.

\begin{lemma}\label{lemma-for-main-thm-i}
 %For any 2-complex $K$ with fundamental group $\pi$,
 The map $\red_2\colon H^2(K;\Z\pi) \to H^2(K;\Z/2)$
 % induced by $\red_2$
 is surjective.
\end{lemma}

\begin{proof}
  This follows from Bockstein sequence associated with $0 \to \ker(\red_2) \to \Z\pi \xrightarrow{\red_2} \Z/2 \to 0$, using that $H^3(K;\ker(\red_2))=0$ since $K$ is 2-dimensional.
\end{proof}

We move on to the proof of \cref{thm:main-technical-thm-intro}~\eqref{it:main1}. It states:
\emph{The element $PD(f^*(\widehat{\phi})) \in H_2(M;\Z\pi) \cong \pi_2(M)$ is $\RPT$-characteristic and has trivial self-intersection number, so that the Kervaire--Milnor invariant $\tau(PD( f^*(\widehat{\phi}))) \in \Z/2$ is well-defined.}

Let $M^{(3)}$ be the 3-skeleton of $M$ for some chosen handle decomposition.
The following lemma is more general than needed in the paper, since it starts with $g \colon M^{(3)}\to K$ only defined on the 3-skeleton, but we give the generalisation here since we appeal to it in \cite{KPT20}.

\begin{lemma}
	\label{lem:char}
	Let $g\colon M^{(3)}\to K$ be a map that is an isomorphism on fundamental groups. Let $j\colon M^{(3)}\to M$ be the inclusion of the $3$-skeleton. For every $\phi\in \ker \Sq^2\subseteq H^2(\pi;\Z/2)$ and every lift $\wh\phi\in H^2(K;\Z\pi)$ of $i^*\phi \in H^2(K;\Z/2)$, the element
\[(PD\circ (j^*)^{-1}\circ g^*)(\wh \phi) \in H_2(M;\Z\pi) \cong \pi_2(M)\]
is $\RPT$-characteristic.
\end{lemma}

In \cref{lem:char} we used the following sequence of maps:
\[\wh\phi \in H^2(K;\Z\pi) \xrightarrow{g^*} H^2(M^{(3)};\Z\pi) \xrightarrow{(j^*)^{-1}} H^2(M;\Z\pi) \xrightarrow{PD} H_2(M;\Z\pi) \cong \pi_2(M).\]

\begin{proof}
	Fix a map $\beta\colon \RP^2\to M$.
	%As before, let $\red_2 \colon \Z\pi \to \Z \to \Z/2$ be the ring homomorphism given by the composition of the augmentation and reduction modulo two.
    Let $c \colon M \to B\pi$ be a $2$-connected map such that $i_* \circ g_* = c_* \circ j_* \colon \pi_1(M^{(3)}) \to \pi_1(B\pi)=\pi$, from which it follows that $i \circ g$ and  $c \circ j$ are homotopic.
	The following equalities prove that $PD((j^*)^{-1}g^*\wh\phi)\in \pi_2(M)$ is $\RPT$-characteristic.  We will give justification for each of the equalities afterwards.
 %\commentm{Added an extra term so that now all definitions are respected.}
		\[\begin{array}{rccl}
		  &\lambda_2(\red_2(PD((j^*)^{-1}g^*(\wh\phi))),\beta_*[\RP^2]) &=&\lambda_2(\beta_*[\RP^2],\red_2(PD((j^*)^{-1}g^*(\wh\phi)))) \\
		=&\langle \red_2((j^*)^{-1}g^*(\wh\phi)),\beta_*[\RP^2]\rangle &=&\langle (j^*)^{-1}g^*\red_2(\wh\phi),\beta_*[\RP^2]\rangle \\
		=&\langle (j^*)^{-1}g^*i^*(\phi),\beta_*[\RP^2]\rangle &=&\langle (j^*)^{-1}j^*c^*(\phi),\beta_*[\RP^2]\rangle \\
		=&\langle c^*(\phi),\beta_*[\RP^2]\rangle &=& \langle \beta^*c^*(\phi),[\RP^2]\rangle.
%		=&\langle \phi, c_*j_*j_*^{-1}\beta_*[\RP^2]\rangle
%		&=&\langle \phi, c_*\beta_*[\RP^2]\rangle=\langle \beta^*c^*\phi,[\RP^2]\rangle
	\end{array}\]
	The first equality uses symmetry of $\lambda_2$. The second equality is the algebraic definition of the intersection form. The third equality uses the naturality of the reduction mod 2. The fourth equality uses that, by definition of $\wh\phi$, $\red_2(\wh\phi) = i^*\phi$. The fifth equality uses that $i \circ g$ and $c \circ j$ are homotopic.
The last equality uses the naturality of the evaluation.
	
Using obstruction theory and the fact that $B\pi$ is aspherical, the map $c\circ \beta\colon \RP^2\to B\pi$ extends to a map $\beta'\colon \RP^\infty\to B\pi$. Now assume for a contradiction that $\langle \beta^*c^*(\phi),[\RP^2]\rangle$ is nontrivial. Then $\beta^*c^*(\phi)\in H^2(\RP^2;\Z/2)$ is nontrivial and hence also $(\beta')^*(\phi)\in H^2(\RP^\infty;\Z/2)$ is nontrivial. In this case, also $\Sq^2((\beta')^*(\phi))=(\beta')^*\Sq^2(\phi)\in H^4(\RP^\infty;\Z/2)$ has to be nontrivial. But we assumed that $\phi$ lies in the kernel of $\Sq^2$. Hence $\langle \beta^*c^*(\phi),[\RP^2]\rangle$ has to be trivial, as desired.
\end{proof}

As explained in \cref{section:tau-spheres}, the Kervaire--Milnor invariant is well-defined on $\RPT$-characteristic spheres with vanishing self-intersection number, so \cref{thm:main-technical-thm-intro}~\eqref{it:main1} follows from the next two lemmas.

\begin{lemma}\label{lem:spherically-char}
	The element $PD(f^*(\wh \phi)) \in H_2(M;\Z\pi) =\pi_2(M)$ is $\RPT$-characteristic.
\end{lemma}

\begin{proof}
	This follows directly from \cref{lem:char} applied to the composition $f \circ j \colon M^{(3)} \to K$, where as in that lemma $j \colon M^{(3)} \to M$ is the inclusion of the $3$-skeleton.
\end{proof}

\begin{lemma}\label{lem:self-int-zero}
	For every $x,y \in H^2(K;\Z\pi)$, we have that $\lambda_M(PD(f^*(x)),PD(f^*(y)))=0$.
	In particular, the self-intersection number vanishes: $\mu(PD(f^*(x)))=0$.
\end{lemma}

\begin{proof}%[Proof of \cref{lem:self-int-zero}]
	Since $K$ is $2$-dimensional we have $f_*([M])=0$, and thus
	\begin{align*}
		\lambda(PD(f^*(x)),PD(f^*(y)))&=\langle f^*(y),PD(f^*(x))\rangle=\langle y,f_*(f^*(x)\cap [M])\rangle\\&=\langle y,x\cap f_*([M])\rangle=\langle y,0\rangle=0 .
\end{align*}
The last sentence in the statement, that $\mu(PD(f^*(x)))=0$, now follows from \cref{lem:lambda-zero-mu-zero}.
\end{proof}

This completes the proof of \cref{thm:main-technical-thm-intro}~\eqref{it:main1}.
The methods used in its proof also allow us to prove \cref{thm1}.
Recall that we have to show:
\emph{(i) For every element $x \in \ker \Sq$, $PD(f^*(x)) \in \pi_2(M)$  has trivial self-intersection number and is $\RPT$-characteristic.
(ii) The induced map $\tau_{M,f} \colon \ker \Sq \to \Z/2$ factors through $\Z/2\otimes_{\Z\pi} \ker \Sq $. (iii) $\tau_{M,f}$, up to the action of $\Aut(\pi)$ on the choice of $f$, is a stable diffeomorphism invariant.}

\begin{proof}[Proof of \cref{thm1}]
	By \cref{lem:self-int-zero} every element in the radical $\Rad(\lambda_M)$ has trivial self-intersection number. We will now show that every element in $PD(f^*(\ker \Sq))$ is also $\RPT$-characteristic. Let $x\in \ker \Sq \subseteq H^2(\pi;\Z\pi)$ and fix a map $\beta\colon \RP^2\to M$.  As in the proof of \cref{lem:char},
\[\lambda_2(\red_2(PD(f^*(x))),\beta_*[\RPT])= \langle \red_2(f^*(x)),\beta_*[\RPT]\rangle = \langle \red_2(\beta^*f^*(x)),[\RPT]\rangle.\]
Let $\beta'\colon \RP^\infty\to B\pi$ be an extension of $f\circ \beta$.
We continue to follow the pattern of the proof of \cref{lem:char}.  Assume for a contradiction that $\langle \red_2(\beta^*f^*(x)),[\RPT]\rangle$ is nontrivial. Then $\red_2(\beta^*f^*(x)) \neq 0 \in H^2(\RPT;\Z/2)$, so $\red_2 \circ (\beta')^*(x)$ is nontrivial in $H^2(\RP^{\infty};\Z/2)$. Therefore
\[(\beta')^*(\Sq(x)) = (\beta')^*(\Sq^2 \circ \red_2(x)) = \Sq^2 \circ \red_2((\beta')^*(x))  \in H^4(\RP^{\infty};\Z/2)\]
is also nontrivial.  This contradicts that $x \in \ker(\Sq)$.
We deduce that $\langle \red_2(\beta^*f^*x),[\RP^2]\rangle$ vanishes, so that $PD(f^*(x))$ is $\RPT$-characteristic as desired, proving~\eqref{thm1:item1}.
	
Since stabilisation does not change the value of $\tau$, it follows that $\tau_{M,f}$ is a stable diffeomorphism invariant, up to the choice of $f$, as asserted in \eqref{thm1:item3}.

 It remains to show \eqref{thm1:item2}, that $\tau_{M,f}$ factors through $\Z/2\otimes_{\Z\pi}\ker \Sq$. Assume $[x]=[y]\in \Z/2\otimes_{\Z\pi}\ker \Sq$. Then there are $z_i\in \ker \Sq$ and $\kappa_i\in \ker(\red_2\colon\Z\pi\to \Z/2)$ such that $x=y+\sum_{i=1}^n\kappa_iz_i$. As $\mu(y)=\mu(z_i)=\lambda(y,z_i)=0$, it follows from \cite[Lemma~8.3]{KLPT15} that $\tau(y)=\tau(y+\sum_{i=1}^n\kappa_i z_i)=\tau(x)$.
\end{proof}

Now we continue with the proof of \cref{thm:main-technical-thm-intro}. We have proven \cref{thm:main-technical-thm-intro}~\eqref{it:main1}, and so we may define $\tau(PD(f^*(\widehat{\phi}))) \in \Z/2$.
We will prove \cref{thm:main-technical-thm-intro}~\eqref{it:main2} and \eqref{it:main3} by comparing the Kervaire--Milnor invariant to the Arf invariant.

Recall that we have to show the following.
\emph{
\begin{enumerate}
\item[(iii)] The map  $\wh{\tau}_{M,f} \colon \ker \Sq^2 \to \Z/2 ;\;\; \phi  \mapsto \tau(PD(f^*(\widehat{\phi})))$
is a well-defined homomorphism.
\item[(iv)] Under the map
$\Hom(\ker \Sq^2,\Z/2) \xrightarrow{\cong} H_2(\pi;\Z/2)/\im \Sq_2 \to H_2(\pi;\Z/2)/\im (d_2,d_3),$
$\wh\tau_{M,f}$ is sent to $\ter(M)$.
\end{enumerate}
}

\begin{remark}
Let us explain the strategy of the rest of the proof.
It follows from \cref{lem:arf,lem:arf-homo-plus-computes-ter} that computing the Arf invariant $\Arf(f^*i^*(\phi))$ gives rise to a well-defined homomorphism in $\Hom(\ker \Sq^2,\Z/2)$ that determines $\ter(M)$.
Thus, to show \eqref{it:main2} and \eqref{it:main3}, we shall prove that the Arf invariant $\Arf(f^*i^*(\phi))$ coincides with the Kervaire--Milnor invariant $\tau(PD(f^*(\wh\phi)))$, where as before $\wh\phi\in H^2(K;\Z\pi)$ is a lift of $i^*\phi\in H^2(K;\Z/2)$. For this we will use the description of the Kervaire--Milnor invariant for $\pi_1$-trivial (embedded) surfaces from \cref{section:tau-surfaces}.

For a $\pi_1$-trivial, closed, oriented, generically immersed surface $F \colon \Sigma \looparrowright M$, the definition of $\tau(F)$ uses a quadratic refinement $\varpi \colon H_1(\Sigma;\Z/2) \to \Z/2$ of the $\Z/2$-intersection form of $\Sigma$ that uses, for each curve $\gamma$ on $\Sigma$, a relative Euler number and a count of intersections arising from $F(\gamma)$.

For the specific $F$ that arise in our situation, which will be embedded, we will show that after picking a correctly framed disc, the relative Euler number in the definition of $\varpi$ agrees with $\Upsilon$. (Recall that $\Upsilon$ is the quadratic refinement we use for computing the Arf invariant.) Then we will show that for such a disc the intersection component of the quadratic refinement $\varpi$ is always even, so does not contribute to the calculation of $\varpi$. It will follow that $\varpi$ and $\Upsilon$ coincide, and therefore the Arf invariant and the Kervaire--Milnor invariant of $F$ agree.
\end{remark}

Next, we want to produce, for each $\wh{\phi} \in H^2(K;\Z\pi)$, a suitable surface $F$ on which to compare the two invariants.
Recall that $H^2(K;\Z\pi)$ is isomorphic to the compactly supported cohomology $H^2_{\cs}(\wt K;\Z)$ of the universal cover $\wt K$ of $K$.
An element $\wh\phi\in H^2(K;\Z\pi)\cong H^2_{\cs}(\widetilde K;\Z)$ can be represented as a map $\wh\phi\colon \wt{K}\to S^2$ with compact support (i.e.\ the closure of the inverse image of $S^2 \sm \{\ast\}$ is compact). This follows from the fact that $\wt{K}$ is $2$-dimensional, and the definition of cohomology with compact support as a colimit of $H^2(\wt{K}, \wt{K} \sm L)$ over compact subsets $L \subseteq \wt{K}$.

Let $\wt{f} \colon \wt{M} \to \wt{K}$ be a lift of the map $f \colon M \to K$.
We want to consider $\wt{f}^*\wh\phi$ as an element of $H_{\cs}^2(\wt M;\Z)$, and thus we need the following lemma.

\begin{lemma}\label{lemma:wt-f-proper}
    The map $\wt f$ is proper, i.e.\ closed and preimages of compact sets are compact.
\end{lemma}

\begin{proof}
    Since $M$ is compact and both $M$ and $K$ are Hausdorff,  standard arguments show that the map $f$ is proper. This implies that $\wt f$ is proper since pullbacks of proper maps are proper. This is well-known but we give a short proof. Recall that a map $h$ is called \emph{universally closed} if every pullback of $h$ is closed. By \cite[\href{https://stacks.math.columbia.edu/tag/005R}{Tag 005R}]{stacks-project}, a map is proper if and only if it is univerally closed. Hence it remains to show that pullbacks of universally closed maps are universally closed. If $g$ is a pullback of a universally closed map $h$, then every pullback of $g$ is also a pullback of $h$ and hence is closed. It follows that $g$ is universally closed, as needed.
\end{proof}

%Then, as we will prove in the next lemma, a surface $F \colon \Sigma \to \wt{M}$ with  \[F(\Sigma) = (\wh\phi\circ\wt{f})^{-1}(x)\subseteq \wt{M}\] represents \[PD(\wt{f}^*(\wh\phi))\in H_2(\wt{M};\Z).\]

\begin{lemma}\label{lem:inverse-image-representative}
	Let $x\in S^2$ be a regular value of $\wh\phi\circ \wt{f}\colon \wt{M}\to S^2$.
 Then the inverse image of $x \in S^2$ yields a surface $\wh{F} \colon \Sigma \to \wt{M}$ such that $\wh{F}(\Sigma)$ represents  $PD(\wt{f}^*(\wh\phi)) \in H_2(\wt{M};\Z)$.
\end{lemma}

\begin{proof}%[Proof of \cref{lem:inverse-image-representative}]
	We start with a general remark about the compact case.
 For $Y$ a compact $4$-manifold, each cohomology class $y$ in $H^2(Y,\partial Y;\Z)$ can be represented by a map $h_y \colon (Y,\partial Y) \to (\mathbb{CP}^2,*)$ (upon post-composing with the inclusion $\CP^2 \to \CP^\infty \simeq K(\Z,2)$), and the inverse image of $\mathbb{CP}^1 \subseteq \mathbb{CP}^2$ is the Poincar\'{e} dual to the original class~$y$.
	%The strategy of the proof is to deduce the lemma from the compact case by recalling that duality for non-compact manifolds can be expressed in terms of the duality maps for the submanifolds in a compact exhaustion by codimension zero submanifolds.

Since $\wt f$ is proper by \cref{lemma:wt-f-proper}, and $\wh{\phi}$ has compact support, it follows that $\wh \phi \circ \wt f\colon \wt M\to S^2$ has compact support. Take the composition with the inclusion $\iota \colon S^2 \to \CP^2$. This yields a map $\iota \circ \wh \phi \circ \wt f \colon \wt{M} \to \CP^2$ representing
$\wt{f}^*(\wh{\phi}) \in H^2_{\cs}(\wt{M};\Z)$.
We have an isomorphism
\[H^2_{\cs}(\wt{M};\Z) = \colim_{L} H^2(\wt{M},\wt{M} \sm L;\Z)\]
where $L$ belongs to the collection of compact subsets of $\wt{M}$ ordered by inclusions.
Our class $\wt{f}^*(\wh{\phi})$ is represented in the colimit by an element of $H^2(\wt{M},\wt{M} \sm L;\Z)$,  for some  compact set $L$ containing the support of $\iota \circ \wh \phi \circ \wt f$. We may and shall assume that $L$ is a compact codimension zero submanifold in~$\wt{M}$. We consider the maps shown in the next diagram.
\[
\begin{tikzcd}[row sep=small, column sep=small]
	H^2_{\cs}(\wt{M};\Z) & \ar[l] H^2(\wt{M},\wt{M} \sm L;\Z) \ar[r,"\cong","\operatorname{exc}"'] & H^2(L,\partial L;\Z) \ar[r,"\cong","PD"'] &
 H_2(L;\Z) \ar[r,"\operatorname{incl}"']&   H_2(\wt{M};\Z)\\
\wt{f}^*(\wh{\phi}) \arrow[u, phantom, sloped, "\in"] & \iota \circ \wh \phi \circ \wt f \ar[r,mapsto] \ar[l,mapsto] \arrow[u, phantom, sloped, "\in"]& \iota \circ \wh \phi \circ \wt f|_L \ar[r,mapsto]\arrow[u, phantom, sloped, "\in"] & (\iota\circ\wh \phi \circ \wt f)|_{L}^{-1}(\CP^1) \ar[r,mapsto]\arrow[u, phantom, sloped, "\in"] & (\iota\circ\wh \phi \circ \wt f)^{-1}(\CP^1) \arrow[u, phantom, sloped, "\in"]
\end{tikzcd}
\]
Consider $\iota \circ \wh \phi \circ \wt f \in H^2(\wt{M},\wt{M} \sm L;\Z)$. We argued above that this maps to the left to $\wt{f}^*(\wh{\phi})$. By naturality of cap products, $PD(\wt{f}^*(\wh{\phi}))$ is given by the image of $\iota \circ \wh \phi \circ \wt f$ in $H_2(\wt{M};\Z)$,  under the composition shown.  In the diagram we claim that this image is $(\iota\circ\wh \phi \circ \wt f)^{-1}(\CP^1)$.
To see this, use the fact for compact manifolds from the first paragraph of the proof, to see that Poincar\'e duality maps  $\iota \circ \wh \phi \circ \wt f|_L$ to the inverse image of $\mathbb{CP}^1 \subseteq \mathbb{CP}^2$, as shown. Since $\iota\circ\wh \phi \circ \wt f$ is supported in $L$, the claim follows.
We see that
	\[PD(\wt{f}^*(\wh{\phi})) = (\iota\circ\wh \phi \circ \wt f)^{-1}(\CP^1) = (\wh \phi \circ \wt f)^{-1}(x),\]
 which proves the desired statement.
 Here for the second equality we use that, since the map $\wh\phi\circ \wt{f} \colon \wt{M} \to \mathbb{CP}^2$ factors through $S^2$, the inverse image of a generic $\mathbb{CP}^1$ is the inverse image of the point $x\in S^2$.
\end{proof}

\begin{lemma}\label{lemma:F-has-right-properties}
After potentially perturbing the map $\wh\phi \colon \wt{K} \to S^2$, the composition $F\colon \Sigma \xrightarrow{\wh F} \wt{M} \to M$ is a $\pi_1$-trivial embedding with $\mu(F)=0$ and such that $[F]$ is $\RPT$-characteristic.
\end{lemma}
\begin{proof}
Note that $\wh F$ is an embedding, since it is the inverse image of the point $x \in S^2$. Since $\pi_1(\wt{M})=0$, $\wh F$ is $\pi_1$-trivial.
We can perturb the map $\wh\phi \colon \wt{K} \to S^2$ so that $(\wh\phi)^{-1}(x) \subseteq \wt{K}$ is a finite (coming from compact support) discrete set, which satisfies that no two points of $(\wh\phi)^{-1}(x)$ have the same image under $\wt{K}\to K$.  Then $F$ is still an embedded $\pi_1$-trivial surface, representing a class $[F] \in H_2(M;\Z\pi) \cong \pi_2(M)$.
By \cref{lem:spherically-char}, $[F]$ is $\RPT$-characteristic.  By \cref{lem:self-int-zero}, $\mu(F)=0$.
\end{proof}

\begin{definition}
	Let $(v_1,v_2)\in T_xS^2 \oplus T_xS^2$ be a framing of the point $x$. For each simple closed curve $\alpha$ in $\Sigma$ we pick a generically immersed disc $C_\alpha$ in $M$ with boundary $F(\alpha)$, such that the image of the normal vector of $S^1\subseteq D^2$ in
 \[T_{C_\alpha(y)}M\cong DF (T_{C_\alpha(y)}\Sigma) \oplus \nu_{F}^{M}|_{C_\alpha(y)}\]
 agrees with $(0,(\wh\phi\circ f)^*v_1)$ for every $y\in S^1$.  We can construct such a disc $C_{\alpha}$ by taking an annulus $S^1 \times I \subseteq M$ with $S^1 \times \{0\} = F(\alpha)$ and such that a nonvanishing section of $\nu_{S^1 \times \{0\}}^{S^1 \times I}$, pushed forward into $TM$, agrees with $(0,(\wh\phi\circ f)^*v_1)$. Then cap off $S^1 \times \{1\}$ with the trace of a null-homotopy in $M$.  We say that $C_\alpha$ is an \emph{$f$-cap} for $\alpha$.
\end{definition}

\begin{lemma}\label{lem:spin-bordism-equals-Euler-no}
	Let $\alpha$ be a simple closed curve on $\Sigma$ and let $C_{\alpha}$ be an $f$-cap for $\alpha$. The spin bordism class of $\alpha$, as an element of $\Omega_1^{\spin} \cong \Z/2$, is equal to the relative Euler number $e(C_\alpha)$ of $C_\alpha$.
\end{lemma}

\begin{proof}
	The bundle $\nu_{F(\alpha)}^{F(\Sigma)}$ is $1$-dimensional, and thus we obtain a canonical framing $w$  from the orientation. The framing $(w,(\wh\phi\circ f)^*v_1,(\wh\phi\circ f)^*v_2)$ of $\nu_{F(\alpha)}^M$ together with the given spin structure on $\nu_M^{\R^\infty}|_{F(\alpha)}$ determines the spin structure on $\alpha$.
	
	As $\nu_M^{\R^\infty}|_\alpha$ extends over $C_\alpha$ and $\nu_{F(\alpha)}^{C_\alpha}$ agrees with $(\wh\phi\circ f)^*v_1$, $\alpha$ is spin null bordant if and only if the framing $(w,(\wh\phi\circ f)^*v_2)$ stably extends over $C_\alpha$. Since $\nu_{C_\alpha}^M$ is 2-dimensional, the normal vector $w$ extends over ${C_\alpha}$ if and only if $(w,(\wh\phi\circ f)^*v_2)$ extends over ${C_\alpha}$. Thus $(w,(\wh\phi\circ f)^*v_2)$ can stably be extended over ${C_\alpha}$ if and only if the relative Euler number $e(C_\alpha)$ is even, so is zero modulo $2$.  This completes the proof of the lemma.
\end{proof}

We have one final lemma for the proof of \cref{thm:main-technical-thm-intro}.

\begin{lemma}\label{lem:int-points-even}
	Let $\alpha$ be a simple closed curve on $F$ and let $C_{\alpha}$ be an $f$-cap for $\alpha$. The interior of the image of~${C_\alpha}$ intersects~$F$ transversely in an even number of points.
\end{lemma}

\begin{proof}%[Proof of \cref{lem:int-points-even}]
	The image of the boundary $S^1$ under $f\circ {C_\alpha} \colon D^2 \to K$ is a point. Thus $f\circ {C_\alpha}$ factors as $f\circ {C_\alpha}\colon D^2 \to S^2\xrightarrow{j} K$, where $j$ is defined by this factorisation.
	
	Recall that we have a map $\wh\phi \colon \wt{K} \to S^2$ representing $\wh\phi \in H^2_{cs}(\wt{K};\Z) \cong H^2(K;\Z\pi)$ that lifts $i^*\phi \in H^2(K;\Z/2)$, where $i \colon K \to B\pi$ is the inclusion of the 2-skeleton.  Let $p\colon \wt{K}\to K$ be the projection and define a map $\psi\colon K\to S^2$ that sends the points $p((\wh\phi)^{-1}(x))$ to $x$ and sends everything outside a small neighbourhood of these points to the base point of $S^2$.
 Choose a model for $K(\Z/2,2)$ with 2-skeleton $S^2$.
 Then $\psi\colon K\to S^2$ composed with the inclusion $\varrho \colon S^2 \to K(\Z/2,2)$ represents $i^*\phi \in H^2(K;\Z/2)$.
	
	Since the normal bundle of $S^1 \subseteq D^2$ under $D{C_\alpha} \colon TD^2 \to TM$ agrees with the direction of $(\wh\phi\circ f)^*v_1$, and $F(\Sigma)$ is the preimage of the points $p((\wh\phi)^{-1}(x))$, the mapping degree of $\psi\circ j \colon S^2\to S^2$ agrees with the number of transverse intersections of $\mathring{C}_\alpha$ with $F$.
	
	Compose $\psi \circ j \colon S^2 \to S^2$ with the inclusion of the $2$-skeleton $\varrho \colon S^2 \to K(\bbZ/2,2)$. The map $\varrho \circ \psi\circ j \colon S^2 \to K(\Z/2,2)$ factors through $B\pi$ by definition of $\psi= i^*\phi$, and therefore is null homotopic, since $B\pi$ is aspherical. It follows that the mapping degree of $\psi\circ j$ is even, which proves the lemma.
\end{proof}

\begin{proof}[Proof of \cref{thm:main-technical-thm-intro}]
	As already mentioned, \eqref{it:main1} follows directly from \cref{lem:spherically-char} and \cref{lem:self-int-zero}.

 By \cref{lem:arf,lem:arf-homo-plus-computes-ter}, $\ter(M)$ can be computed using the codimension two Arf invariant and the latter determines a homomorphism.
 Hence to prove \eqref{it:main2} and \eqref{it:main3} it suffices to show $\Arf(F)=\tau(F)$ for the surface $F$ defined in \cref{lemma:F-has-right-properties}: it will follow that $\wh{\tau}_{M,f}$ is a homomorphism and that this homomorphism maps to $\ter(M)$ under the composition displayed in \cref{thm:main-technical-thm-intro}~\eqref{it:main3}. In that lemma we also showed that $F$ is $\pi_1$-trivial, $\RPT$-characteristic and $\mu(F)=0$. Therefore we can compute and compare both $\Arf(F)$ and $\tau(F)$.

	The Arf invariant of $F$ depends only on the relative Euler numbers of the $f$-caps by \cref{lem:spin-bordism-equals-Euler-no},
%discs bounding curves on $F$,
	whereas the $\tau$ invariant depends on the relative Euler number \emph{and} the intersections of the form $\mathring{C} \pitchfork F$.  \cref{lem:int-points-even} shows that the latter do not contribute to the calculation of the $\tau$ invariant.
	Therefore we have $\Arf(F)=\tau(F)$, as desired, which completes the proof.
\end{proof}

\bibliographystyle{alpha}
\bibliography{classification}
\end{document}